\documentclass[12pt]{amsart}
\usepackage[pdftex,colorlinks,citecolor=blue,urlcolor=blue]{hyperref}
\usepackage{amssymb}
\usepackage{graphicx,color}
\usepackage{amsmath}
\usepackage{comment}
\usepackage{MnSymbol}
\usepackage{cleveref}
\numberwithin{equation}{section}

\newtheorem{theorem}{Theorem}[section]
\newtheorem{proposition}[theorem]{Proposition}
\newtheorem{question}[theorem]{Question}
\newtheorem{corollary}[theorem]{Corollary}

\newtheorem{lemma}[theorem]{Lemma}

\newtheorem{problem}[theorem]{Problem}

\theoremstyle{definition}
\newtheorem{definition}[theorem]{Definition}
\newtheorem{example}[theorem]{Example}
\newtheorem{remark}[theorem]{Remark}

\DeclareMathOperator{\lk}{\mathrm{lk}}
\DeclareMathOperator{\st}{\mathrm{st}}

\DeclareMathOperator{\Sp}{\mathbb{S}}

\newcommand{\field}{{\bf k}}

\newcommand{\R}{{\mathbb R}}

\newcommand{\Z}{{\mathbb Z}}

\newcommand{\pl}{\overset{\mathrm{PL}}{\cong}}

\title{A new family of triangulations of $\mathbb{R}P^d$}

\author[L. Venturello]{Lorenzo Venturello}
\email{lorenzo.venturello@mis.mpg.de}
\address{
	Max Planck Institute for Mathematics in the Sciences,
	Inselstr. 22,
	04103 Leipzig, GERMANY.
}
\author[H. Zheng]{Hailun Zheng}
\email{hailunz@umich.edu}
\address{
	Department of Mathematics,
	University of Michigan,
	Ann Arbor, Michigan 48109-1043, USA
}
\date{\today}

\begin{document}
	\maketitle
	\begin{abstract}
		We construct a family of PL triangulations of the $d$-dimensional real projective space $\mathbb{R}P^d$ on $\Theta((\frac{1+\sqrt{5}}{2})^{d+1})$ vertices for every $d\geq 1$. This improves a construction due to K\"{u}hnel on $2^{d+1}-1$ vertices. 
	\end{abstract}
	\section{Introduction and main results}
	Triangulations of topological spaces play an important role in many areas of mathematics, from the more theoretical to the applied ones. 
	A classical problem in PL topology asks for the minimum number of vertices that a simplicial complex with a certain geometric realization can have. This invariant depends on the underlying topological space: triangulable spaces with complicated homology or homotopy groups tend to need more vertices in their vertex-minimal triangulations. However, to determine this number is in general very hard, even if we restrict our discussion to manifolds or \emph{PL manifolds}. One of the few families for which this number is known is that of sphere bundles over the circle. K\"{u}hnel \cite{Kuhnel86} showed that the boundary complex of the $(2d+3)$-vertex stacked $(d+1)$-manifold whose facet-ridge graph is a cycle is a PL triangulation of  $\mathbb{S}^{d-1}\times \mathbb{S}^1$ for even $d$, and a PL triangulation of the twisted bundle $\mathbb{S}^{d-1}\dtimes \mathbb{S}^1$ for odd $d$. Furthermore, K\"uhnel's triangulations are vertex-minimal. In the remaining cases, namely when $d$ is odd and the bundle is orientable, or when $d$ is even and the bundle is non-orientable, the minimum number of vertices is $2d+4$, and it is attained for every $d$ \cite{BagchiDatta-JCTA-08,ChestnutSapirSwartz}. Novik and Swartz \cite{NSw-socle} (in the orientable case) and later Murai \cite{Mur15} proved that the number of vertices $f_0(\Delta)$ of a $\field$-homology $d$-manifold $\Delta$ must satisfy $\binom{f_0(\Delta)-d-1}{2}\geq\binom{d+2}{2}\widetilde{\beta}_1(\Delta; \field)$, for $d\geq 3$ and any field $\field$. Equality in this formula is attained precisely by members in the \emph{Walkup class} which are moreover $2$-neighborly, i.e, their graphs are complete. 
	In low dimensions computational methods are a precious source. Using the program BISTELLAR \cite{BjLu} Bj\"{o}rner and Lutz constructed several vertex-minimal triangulations of $3$- and $4$-dimensional manifolds, which are collected in \cite{LutThesis}. 
	
	For a manifold whose homology (computed with coefficients in $\Z$) has a nontrivial torsion part, the Novik-Swartz-Murai lower bound is far from being tight.
	In this article we focus on the triangulations of the real projective space $\mathbb{R}P^d$. Arnoux and Marin \cite{ArMa} gave a lower bound on the number of vertices of a triangulation of $\mathbb{R}P^d$.
	\begin{theorem}\cite{ArMa}\label{thm: lower b}
		Let $\Delta$ be a triangulation of $\mathbb{R}P^d$, with $d\geq 3$. Then
		$$f_0(\Delta)\geq\binom{d+2}{2}+1.$$
	\end{theorem}
	It is well known that there is a unique vertex-minimal triangulation of $\mathbb{R}P^2$ on $6$ vertices, and Walkup \cite{Walkup} proved that the number of vertices needed to triangulate $\mathbb{R}P^3$ is at least $11$, and constructed one such complex. More recently, Sulanke and Lutz obtained an enumeration of manifolds on $11$ vertices, revealing $30$ non-isomorphic such triangulations, exhibiting $5$ different $f$-vectors (see \cite[Table 10]{SuLu}). This also shows that Walkup's construction is the \emph{unique} $f$-vectorwise minimal triangulation of $\R P^3$. Computer search revealed a vertex-minimal triangulation of $\mathbb{R}P^4$ on $16$ vertices. This highly symmetric simplicial complex was studied by Balagopalan  \cite{Bal} who described three different ways to construct it. Even though the $3$- and $4$-dimensional cases suggest the formula in \Cref{thm: lower b} is tight in $d=3,4$, for $d=5$ the computer search could not find triangulations of $\mathbb{R}P^5$ on less than $24$ vertices. It is now very tempting to conjecture a tight lower bound of $\binom{d+2}{2}+\lfloor\frac{d-1}{2}\rfloor$, which fits all the known cases and reflects the fact that the number of nontrivial integral homology groups of $\R P^d$ depends on the parity of $d$. Unfortunately, in higher dimensions we don't know any triangulation of $\R P^d$ on $O(d^2)$ or even $O(d^i)$ vertices, for any $i$.  The current record is due to K\"{u}hnel \cite{Kuhnel87}. He observed that the barycentric subdivision of the boundary of the $(d+1)$-simplex possesses a \emph{free involution}, and the quotient w.r.t. the involution is PL homeomorphic to the $d$-dimensional real projective space. This construction provides a PL triangulation of $\mathbb{R}P^d$ on $2^{d+1}-1$ vertices. Indeed, a natural way to construct a triangulation of $\R P^d$ is to find a centrally symmetric (or cs, for short) triangulation of $\Sp^d$, the double cover of $\R P^d$, with an additional combinatorial condition. Let $\Delta$ be a cs simplicial complex $\Delta$ with free involution $\sigma$. We say an $n$-cycle $C$ in $\Delta$ is an \emph{induced cs cycle} if $C=\sigma(C)$ and any face of $\Delta$ with vertices in $C$ must be a face of $C$. 
	\begin{lemma}[{\cite[Proposition (8.1)]{Walkup}}]\label{lem: from sphere to rpd}
		Let $\Delta$ be a cs PL $d$-sphere with free involution $\sigma$ and with no induced cs $4$-cycle. Then $\Delta/\sigma$ is a PL triangulation of $\R P^d$.
	\end{lemma}
	In this article, we construct a family of PL $d$-spheres as in \Cref{lem: from sphere to rpd} for every $d\geq 0$.
	Let $F_i$ be the $i$-th Fibonacci number, i.e., $F_0=0$, $F_1=1$ and $F_n=F_{n-1}+F_{n-2}$ for every $n\geq2$. Our main result is the following.
	\begin{theorem}\label{thm: exists cs with no cycle}
		There exists a cs PL $d$-sphere $S_d$ with no cs induced $4$-cycles and $f_0(S_d)=3F_{d+1}+7F_d+3F_{d-1}-4$. Consequently, there exists a PL triangulation $\Delta_d$ of $\R P^d$ with $f_0(\Delta_d)=\frac{3}{2}F_{d+1}+\frac{7}{2}F_d+\frac{3}{2}F_{d-1}-2$.   
	\end{theorem}
	The second statement is a direct consequence of \Cref{lem: from sphere to rpd}. Observe that $\frac{3}{2}F_{d+1}+\frac{7}{2}F_d+\frac{3}{2}F_{d-1}-2<2^{d+1}-1$ for every $d\geq 3$, hence improving K\"{u}hnel's construction in any dimension. Moreover the improvement of the bound is asymptotically significant, since $\frac{3}{2}F_{d+1}\leq \frac{3}{2\sqrt{5}}\left( \frac{1+\sqrt{5}}{2}\right)^{d+1}\sim \frac{3}{2\sqrt{5}}(1.61803\dots)^{d+1}$. 
	
	It is worth noting that in combinatorial and computational topology there are many interesting decompositions of topological spaces and manifolds, such as CW complexes, simplicial posets, graph encoded manifolds and more. Consequently, the related literature is vast. Even the expression ``triangulated manifold" may refer to objects other than abstract simplicial complexes as discussed in this paper. It is important to stress that in this article we consider exclusively objects which are simplicial complexes. For results on the minimal triangulations of $\R P^3$ (or more generally, lens spaces and other 3-manifolds) in other settings, see, for example, \cite{BasakDatta, Cas, JRT,Swartz2003}. 
	
	\section{Definitions}\label{sec: 2}
	A \emph{simplicial complex} $\Delta$ with vertex set $V=V(\Delta)$ is a collection of subsets of $V$ that is closed under inclusion. The elements of $\Delta$ are called {\em faces}. For brevity, we usually denote $\{v\}$ as $v$ and with $f_0(\Delta)$ the number $|V(\Delta)|$. The \emph{dimension of a face} $F\in\Delta$ is $\dim F:=|F|-1$. The \emph{dimension of $\Delta$}, $\dim\Delta$, is the maximum dimension of its faces. We let $f_i(\Delta)$ be the number of $i$-dimensional faces of $\Delta$, and record the numbers $(f_0(\Delta),f_1(\Delta),\dots,f_{\dim\Delta}(\Delta))$ in the so called $f$-vector of $\Delta$. 
	
	If $F$ is a face of $\Delta$, then the {\em star of $F$} and the  {\em link of $F$} in $\Delta$ are the simplicial complexes
	\[\st_\Delta(F):=\{\sigma\in\Delta \ : \  \sigma\cup F\in\Delta\} \quad \mbox{and} \quad \lk_\Delta(F):= \{\sigma\in \st_\Delta(F) \ : \ \sigma\cap F=\emptyset\}.
	\]
	If $\Delta$ and $\Gamma$ are simplicial complexes on disjoint vertex sets, then the \textit{join} of $\Delta$ and $\Gamma$ is the simplicial complex $\Delta*\Gamma = \{\sigma \cup \tau \ : \ \sigma \in \Delta \text{ and } \tau \in \Gamma\}$. In particular, if $\Delta_1=\{\emptyset, \{v\}\}$, then $\Delta_1*\Delta_2$ is called the \emph{cone} over $\Delta_2$ with apex $v$.
	
	Let $\Delta$ be a simplicial complex and $W\subseteq V(\Delta)$. The \emph{induced} subcomplex of $\Delta$ on $W$ is $\Delta_W=\{F\in\Delta: \; F\subseteq W\}$. If $\Gamma$  is a subcomplex of $\Delta$, define
	$$\Delta\setminus\Gamma=\Delta_{V(\Delta)\setminus V(\Gamma)}=\{F\in \Delta: F\cap V(\Gamma)=\emptyset\}.$$ 
	The \emph{complement} of $\Gamma$ in $\Delta$, denoted by $\Delta-\Gamma$, is the set of faces in $\Delta$ but not in $\Gamma$. The closure $\overline{\Delta-\Gamma}$ is the subcomplex of $\Delta$ generated by the facets of $\Delta-\Gamma$. If $f$ is an automorphism of $\Delta$, then the \emph{quotient} of $\Delta$ w.r.t. $f$, $\Delta/f$, is the simplicial complex obtained identifying the vertices in the same orbit and all the faces with the same vertex set. Observe that in general $|\Delta/f|\ncong|\Delta|/\tilde{f}$, where $\tilde{f}: |\Delta| \to |\Delta|$ is the continuous map induced by $f$. For instance, if $\Delta$ is a square and $f$ maps every vertex $v$ to the unique vertex not adjacent to $v$ then  $\Delta/f$ is the $1$-dimensional simplex, while $|\Delta|/\tilde{f}\cong \mathbb{S}^1$. 
	
	We say two simplicial complexes $\Delta_1$ and $\Delta_2$ are \emph{PL homeomorphic}, denoted as $\Delta_1 \pl \Delta_2$, if there exist subdivisions $\Delta_1'$ of $\Delta_1$ and $\Delta_2'$ of $\Delta_2$ that are simplicially isomorphic. A \emph{PL $d$-ball} is a simplicial complex PL homeomorphic to a $d$-simplex. Similarly, a \emph{PL $d$-sphere} is a simplicial complex PL homeomorphic to the boundary complex of a $(d+1)$-simplex. A $d$-dimensional simplicial complex $\Delta$ is called a \emph{PL $d$-manifold} if the link of every non-empty face $F$ of $\Delta$ is a $(d-|F|)$-dimensional PL ball or sphere; in the former case, we say $F$ is a \emph{boundary face} while in the latter case $F$ is an \emph{interior face}. The \emph{boundary complex} $\partial \Delta$ is the subcomplex of $\Delta$ that consists of all boundary faces of $\Delta$. A PL manifold whose geometric realization is homeomorphic to a closed manifold $M$ is called a \emph{PL triangulation} of $M$. In the literature PL manifolds as defined above are sometimes called \emph{combinatorial manifolds} or \emph{combinatorial triangulations of manifolds}.
	\begin{remark}
		For $d\geq 5$ the class of PL $d$-manifolds is strictly contained in that of triangulated $d$-manifolds. In particular, the double suspension of any homology $3$-sphere with a non-trivial fundamental group is a non-PL simplicial 5-sphere (see e.g., \cite{RoSa}).
	\end{remark}
	PL manifolds have the following nice properties, see for instance the work of Alexander \cite{Al30} and Lickorish \cite{Lickorish}:
	\begin{lemma}\label{lem: property of PL manifolds}
		Let $\Delta_1$ and $\Delta_2$ be PL $d_1$- and $d_2$-manifolds, respectively.
		\begin{enumerate}
			\item If $\Delta_1$ and $\Delta_2$ are PL balls, so is $\Delta_1*\Delta_2$.
			\item If $d_1=d_2=d$ and $\Gamma:=\Delta_1\cap \Delta_2=\partial\Delta_1\cap \partial \Delta_2$ is a PL $(d-1)$-manifold, then $\Delta_1\cup \Delta_2$ is a PL $d$-manifold. If furthermore $\Delta_2$ and $\Gamma$ are PL balls, then $\Delta_1\cup \Delta_2\pl \Delta_1$.
		\end{enumerate}
	\end{lemma}
	\begin{lemma}(Newman's theorem)\label{lem: newmann}
		Let $\Delta$ be a PL $d$-sphere and let $\Psi\subset\Delta$ be a PL $d$-ball. Then the closure of the complement of $\Psi$ in $\Delta$ is a PL $d$-ball.
	\end{lemma}
	Let $\Delta$ be a PL manifold without boundary. The \emph{prism over $\Delta$} is the pure polyhedral complex $\Delta\times [-1,1]$, whose cells are of the form $F\times\{1\}$,  $F\times\{-1\}$ or $F\times[-1,1]$, for every $F\in\Delta$. Clearly $\left|\Delta\times [-1,1] \right|\cong \left|\Delta\right|\times [-1,1]$. The boundary of the prism consists of the cells $F\times \{1\}$ and $F\times \{-1\}$, where $F\in \Delta$.
	
	A simplicial complex $\Delta$ is \textit{centrally symmetric} or {\em cs} if its vertex set is endowed with a {\em free involution} $\sigma: V(\Delta) \rightarrow V(\Delta)$ that induces a free involution on the set of all non-empty faces of $\Delta$. Let $\Delta$ be a centrally symmetric PL $d$-sphere and let $\sigma$ be the free involution on $\Delta$.  An \emph{induced cs $4$-cycle} in $\Delta$ is an induced subcomplex $C\subseteq\Delta$ isomorphic to a $4$-cycle, with $\sigma(C)=C$. In particular, the vertices of $C$ are $v,w,\sigma(v),\sigma(w)$, for some $v,w\in V(\Delta)$. The complex $\Delta$ has no induced cs $4$-cycle if and only if $\st_{\Delta}(v)\cap\st_{\Delta}(\sigma(v))=\{\emptyset\}$ for every $v\in V(\Delta)$.
	
	In what follows we define a special property in the structure of a cs PL $d$-sphere. We use $\uplus$ to denote either the disjoint union of sets, or the union of two simplicial complexes defined on disjoint vertex sets. 
	\begin{definition}[Property $\mathrm{P}_d$]
		We say that a cs PL $d$-sphere $S$ with free involution $\sigma$ satisfies \emph{Property $\mathrm{P}_d$} if there exists a sequence of subcomplexes $S_0\subset S_1\subset \dots \subset S_d=S$ such that each $S_i$ is a cs PL $i$-sphere with free involution $\sigma|_{S_i}$ that satisfies the following properties:
		\begin{itemize}
			\item[i.] $S_i=B_i\cup\sigma(B_i)$, where $B_i$ are PL $i$-balls and $B\cap\sigma(B)=\partial B = S_{i-1}$. In words, $S_{i-1}$ bounds two PL $i$-balls in $S_i$.
			\item[ii.] $S_i\setminus S_{i-1}= D_i\uplus \sigma(D_i)$, where $D_i$ is a PL $i$-ball. In words, the complement of $S_{i-1}$ in $S_i$ is the disjoint union of two PL $i$-balls.
			\item[iii.] There exists $v_i\in D_i$ such that $V(\st_{S_i}(v_i))\cup V(\st_{S_i}(\sigma(v_i)))=V(S_i)$.		
\end{itemize}
	\label{def: P_d}
	\end{definition}
	Property $\mathrm{P}_d$ describes a family of cs spheres which is the key step of our construction. We shall see in \Cref{prop: no induced 4} that these spheres do not have induced cs $4$-cycles. In the following examples we construct a $1$-sphere and $2$-sphere that satisfy Property $\mathrm{P}_1$ and $\mathrm{P}_2$, respectively.
	\begin{example}\label{ex: S_1}
		Referring to the top left complex in \Cref{fig: prism}, we let $S_0$ be the cs 0-dimensional sphere with vertices $3$ and $\sigma(3)$. We also let $v_1=1$, $D_1=\overline{\{1,2\}}$, and $B_1$ be the path $\sigma(3)-1-2-3$ whose boundary complex is $S_0$. Then the cs 6-cycle $S_1$ satisfies Property $\mathrm{P}_{1}$ w.r.t. $(B_1, D_1, v_1)$.
	\end{example}
	\begin{example}\label{ex: S_0}
		The boundary complex $S_2$ of the icosahedron is a cs $2$-sphere  on $12$ vertices which satisfies Property $\mathrm{P}_{2}$: indeed there is an induced cs 6-cycle $S_1$ in $S_2$ that divides $S_2$ into two antipodal 2-balls. In this case $S_2\setminus S_1$ is the disjoint union of two triangles $D_2, \sigma(D_2)$ and any vertex $v\in D_2$ satisfies the third condition in Definition \ref{def: P_d}.
	\end{example}
	
	\section{From a triangulation of $\Sp^{d-1}$ to a triangulation of $\Sp^{d}$}
	Given a cs PL $(d-1)$-sphere $S_{d-1}$ with $2n$ vertices without any induced cs 4-cycles, one may build a cs PL $d$-sphere with $4n+2$ vertices as follows: first build the prism over $S_{d-1}$; then triangulate the prism in such a way that no interior vertex is created and the central symmetry is preserved; finally cone over the boundaries of the prism (as two disjoint copies of $S_{d-1}$) with two new vertices. The PL $d$-sphere obtained via our construction has no induced 4-cycles. In this section, we will modify the above approach to reduce the number of vertices in the construction.
	
	In what follows, assume that $S_{d-1}$ is a cs PL $(d-1)$-sphere with involution $\sigma$ and it satisfies Property $\mathrm{P}_{d-1}$. In particular, there exists a triple $(B_{d-1},D_{d-1},v_{d-1})$ satisfying the three conditions in \Cref{def: P_d}, and $S_{d-2}:=\partial B_{d-1}$ is a cs PL $(d-2)$-sphere that satisfies Property $\mathrm{P}_{d-2}$ under the triple $(B_{d-2}, D_{d-2}, v_{d-2})$. To make the notation more compact, for the rest of this section we will denote by $S$ the cs $(d-1)$-sphere $S_{d-1}$ and by $(B,D,v)$ the triple $(B_{d-1},D_{d-1},v_{d-1})$. We will first give a triangulation $\Sigma$ of $\mathbb{S}^{d-1}\times [-1,1]$ such that
	\begin{itemize}
		\item $\Sigma$ is centrally symmetric;
		\item there exists a cs subcomplex $\Gamma\subseteq\Sigma$ with $\Gamma\cong S$ and $D\times\{1\}, \sigma(D)\times\{-1\}\subseteq \Gamma$.
	\end{itemize}
	This will be done in two steps. First we construct a polyhedral complex that satisfies the above conditions, see Proposition \ref{prop: sigma}. Then we triangulate it while preserving these properties, see Corollary \ref{cor: exists triangulation prism}. Finally, we build a cs PL $d$-sphere from $\Sigma$, see Corollary \ref{cor: phi is sphere}. 
	\subsection{The prism over $S$}
	The main point of this section is \Cref{prop: sigma}. By \emph{refinement} of a polyhedral complex $P$ we mean a polyhedral complex $\Delta$ with $|\Delta|\cong|P|$ obtained by subdividing $P$.
	\begin{proposition}\label{prop: sigma}
		Let $S$ be a cs PL $(d-1)$-sphere satisfying Property $\mathrm{P}_{d-1}$.
		There exists a refinement $\Sigma''$ of $S\times[-1,1]$ such that:
		\begin{itemize}
			\item $\Sigma''$ is centrally symmetric.
			\item There exists a cs subcomplex $\Gamma\subseteq\Sigma''$ with $\Gamma\cong S$ and $D\times\{1\}, \sigma(D)\times\{-1\}\subseteq \Gamma$.
		\end{itemize}
	\end{proposition}
	After a few intermediate steps, we will prove Proposition \ref{prop: sigma} as a corollary of Lemma \ref{lem: property of Sigma''}. We first refine $S\times[-1,1]$ to a three-layered prism over $S$. Let $\Sigma'$ be the polyhedral complex 
		$$\Sigma':=(S\times[-1,0])\cup (S\times[0,1]).$$
	In other words, $\Sigma'$ is a prism over $S$ with three layers $S\times \{-1\}$, $S\times \{0\}$ and $S\times \{1\}$. Note that $\Sigma'$ is centrally symmetric with the free involution $\sigma'$ induced by $(u,i)\mapsto(\sigma(u),-i)$, for $i=-1,0,1$. Since $S$ satisfies Property $\mathrm{P}_{d-1}$, from conditions ii and iii of \Cref{def: P_d} it follows that
	\begin{equation*}
	\begin{split}
	 V(S)&=V(D)\uplus V(\sigma(D))\uplus V(S_{d-2})\\
	 &= V(D)\uplus V(\st_{S_{d-2}}(v_{d-2})) \uplus  V(\sigma(D))\uplus V(\sigma(\st_{S_{d-2}}(v_{d-2}))).
    \end{split} 
	\end{equation*}
	Hence the following map is well-defined:
	
	\begin{align}
	\psi: V(\Sigma')&\rightarrow V(\Sigma')\nonumber\\
	(u,i)&\mapsto \begin{cases}
	(u,i) & \text{if } i\in\{-1,1\}\\
	(u,1) & \text{if } i=0 \text{ and } u\in V(D)\uplus V(\st_{S_{d-2}}(v_{d-2}))\\
	(u,-1) & \text{if } i=0 \text{ and } u\in V(\sigma(D))\uplus V(\st_{S_{d-2}}(\sigma(v_{d-2})))\\
	\end{cases},
	\label{eq: 3.1}
	\end{align}
	
	Observe that the map $\psi$ induces a map from the polyhedral complex $\Sigma'$ to itself.
	\begin{definition}\label{Sigma'}
		We define
		$$\Sigma'':=\Sigma'/\sim,$$
		where $u\sim w$ if and only if $w=\psi(u)$.
	\end{definition}

	Informally speaking, we identify two disjoint subcomplexes of the middle layer of $\Sigma'$ with the upper and lower layers respectively. In this way we obtain a polyhedral complex containing an induced subcomplex isomorphic to $S$.
	\begin{remark}
		We illustrate these maps in \Cref{fig: prism}, in which case $d=2$.  As in \Cref{ex: S_1}, We let $v_0=3$, $D=\overline{\{1,2\}}$, $B$ be the path $\sigma(3)-1-2-3$, and $S$ be the cs 6-cycle. Following \eqref{eq: 3.1}, the map $\psi$ acts as the identity on the vertices in $S \times \{1\}$ and $S \times \{-1\}$, and maps the vertices of $S\times\{0\}$ to either the upper or lower corresponding vertices. Hence, to obtain the complex $\Sigma''$ from $\Sigma'$ in \Cref{fig: prism}, we identify all pairs of boxed vertices (resp. crossed vertices) lying on the same vertical segments in $S\times [0,1]$ (resp. $S\times [-1,0]$).
	\end{remark}

	\begin{figure}[h]
		\centering
		\includegraphics[scale=0.7]{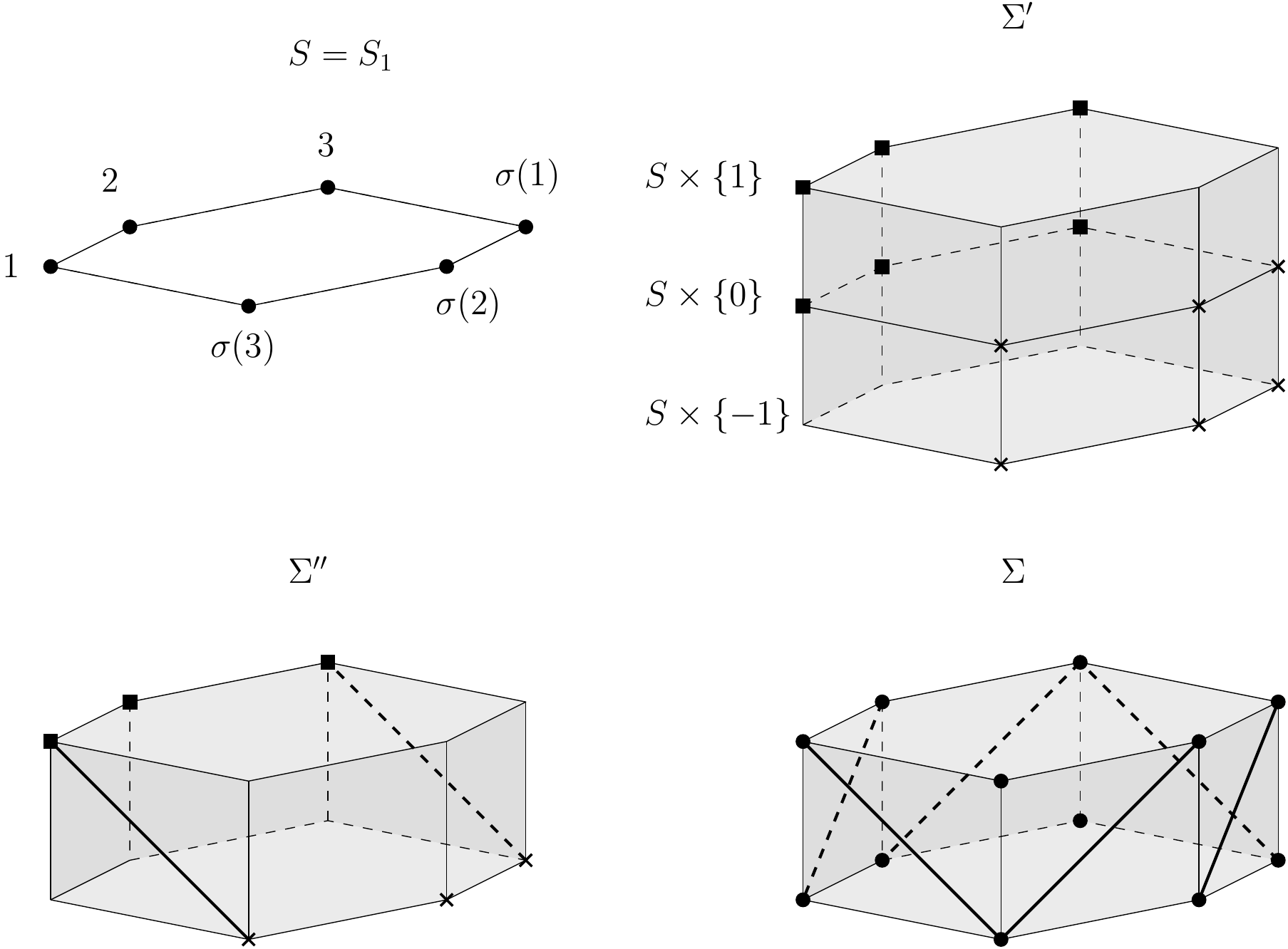}
		\caption{The complexes $\Sigma'$ and $\Sigma''$, together with a cs triangulation $\Sigma$.}
		\label{fig: prism}
	\end{figure}
	
	\begin{lemma}\label{lem: property of Sigma''}
		The complex $\Sigma''$ in \Cref{Sigma'} has the following properties:
		\begin{itemize}
			\item[i.] $\left| \Sigma''\right| \cong \mathbb{S}^{d-1}\times[-1,1]$.
			\item[ii.] $\partial\Sigma''$ consists of two copies of $S$. 
			\item[iii.] There exists an induced cs subcomplex $\Gamma\subseteq\Sigma''$ with $\Gamma\cong S$ and both $D\times\{1\}$ and $\sigma(D)\times\{-1\}$ are subcomplexes of $\Gamma$.
			\item[iv.] $\Sigma''$ is centrally symmetric under the involution $\sigma''$ induced by $\sigma'$.
			
		\end{itemize}
	\end{lemma}
	\begin{proof}
		Part i is clear, and for ii it suffices to observe that the boundary of $\partial \Sigma''$ is exactly the disjoint union of $S\times \{1\}$ and $S\times \{-1\}$.\\
		Let $\Gamma$ be the image of $S\times\{0\}$ in $\Sigma''$. Since the restriction of $\psi$ on $S\times \{0\}$ is injective, $\Gamma=\psi(S\times\{0\})\cong S$. Moreover $D\times\{1\}=\psi(D\times\{0\})\subseteq\Gamma$ and $\sigma(D)\times\{-1\}=\psi(\sigma(D)\times\{0\})\subseteq\Gamma$. This proves iii.\\
		Finally, we observe that if $F=G$ in $\Sigma''$ then $\sigma(F)= \sigma(G)$. Equivalently, the map that assigns the equivalence class of the vertex $(u,i)$ to the class of $\sigma'((u,i))$ is well defined. Therefore, it induces a free involution of $\sigma'':\Sigma''\to \Sigma''$. 
	\end{proof}
	\subsection{The cs triangulation $\Sigma$ of $\Sigma''$}
	In this subsection we construct a \emph{simplicial} complex $\Sigma$ which refines $\Sigma''$ such that 
	\begin{itemize}
		\item $\Sigma$ is a PL manifold with $V(\Sigma)=V(\Sigma'')$, i.e., no new vertex is introduced.
		\item $\Sigma$ is centrally symmetric with a free involution $\tau$, such that $\tau|_{\Sigma''}=\sigma''$. 
	\end{itemize} 
	Our construction is based on certain orientations of the graph of a simplicial complex called locally acyclic orientations. We refer to \cite[Section 7.2]{DeLoera-et-all} for a more detailed treatment of the subject. 
	\begin{definition}
		A \emph{locally acyclic orientation} (l.a.o.) of $\Delta$ is an orientation of the edges of its graph such that none of the $2$-simplices of $\Delta$ contains an oriented cycle.
	\end{definition}
	
	\begin{lemma}[{\cite[Lemma 7.2.9]{DeLoera-et-all}}]\label{triangulation of prism}
		Let $\Delta$ be a simplicial complex. The simplicial refinements of $\Delta\times[-1,1]$ are in bijection with locally acyclic orientations of $\Delta$.
	\end{lemma}
	
	Every simplicial complex has a locally acyclic orientation, obtained for example from a (globally) acyclic orientations of its graph. The bijection in Lemma \ref{triangulation of prism} is easy to describe: if $\{i,j\}$ is an edge of $\Delta$ with $i\rightarrow j$, then $\{(i,1),(j,-1)\}$ is an edge of $\Delta\times[-1,1]$ and vice versa. This induces a triangulation of $F\times [-1,1]$ for every face $F\in \Delta$. Furthermore, the locally acyclicity guarantees that the union of triangulations of individual cells can be coherently completed to a triangulation of $\Delta\times [-1,1]$. 
	
	Given an l.a.o.\ $\ell$ of $\Delta$, we denote by $(\Delta\times [-1,1])^{\ell}$ the refinement of $\Delta\times [-1,1]$ induced by $\ell$. The following lemma shows that if $\Delta$ is a PL sphere, then $(\Delta\times [-1,1])^{\ell}$ is also in the PL category.
	\begin{lemma}\label{lem: is PL manifold}
		Let $\Delta$ be a PL $(d-1)$-sphere and let $\ell$ be any locally acyclic orientation on $\Delta$. Then, the refinement $(\Delta\times [-1, 1])^{\ell}$ is a PL $d$-manifold with boundary and without interior vertices.
	\end{lemma}
	
	\begin{proof}
		Since $\Delta$ is a PL $(d-1)$-sphere it can be subdivided to a complex $\Delta'$ all whose simplices can be linearly embedded in $\R^N$ for some $N$. Moreover, for every locally acyclic orientation $\ell$ of $\Delta$, consider a linear ordering of the vertices of $\Delta'$ such that $v>w$ for every $v\in V(\Delta')\setminus V(\Delta)$ and $w\in V(\Delta)$. Orient all the edges $\{u,v\}\notin \Delta$ with $u\to v$ if $u>v$. This extends to a locally acyclic orientation $\ell'$ of $\Delta'$ that agrees with $\ell$ when restricted to the edges of $\Delta$. The corresponding refinement of $\Delta'\times[-1,1]$ is a subdivision of $\Delta\times[-1,1]$ in which every simplex is linearly embedded in $\R^{N+1}$. Therefore, $(\Delta\times [-1,1])^{\ell}$ is a PL manifold with boundary. Furthermore $V((\Delta\times [-1,1])^\ell)=V(\Delta\times \{-1\})\cup V(\Delta\times \{1\})=V(\partial (\Delta\times [-1,1])^\ell)$ and hence $\Delta^\ell$ has no interior vertices.
	\end{proof}
	
	In what follows, we show that if $\Delta$ is cs, then there exists an l.a.o.\ $\ell$ such that $(\Delta\times [-1,1])^\ell$ is also cs. 
	\begin{lemma}\label{lem: symmetry l.a.o.}
		Let $\Delta$ be a cs simplicial complex with free involution $\sigma$ and consider a refinement of $\Delta\times[-1,1]$ which is cs with free involution induced by $\widetilde{\sigma}:(v,-1) \mapsto (\sigma(v),1), \; (v, 1)\mapsto (\sigma(v), -1)$ for any vertex $v\in \Delta$. Then the corresponding locally acyclic orientation of $\Delta$ is order reversing w.r.t. the symmetry, i.e., $v\rightarrow w$ if and only if $\sigma(w) \rightarrow \sigma(v)$.
	\end{lemma}	 
	\begin{proof}
		In any cs triangulation of $\Delta\times[-1,1]$ the set $\{(v,-1),(w,1)\}$ is an edge if and only if $\{\widetilde{\sigma}((v,-1)),\widetilde{\sigma}((w,1))\}=\{(\sigma(v),1),(\sigma(w),-1)\}$ is an edge, which implies that on the corresponding l.a.o.\ we have that $v\rightarrow w$ if and only if $\sigma(w)\rightarrow \sigma(v)$.
	\end{proof}
	
	\begin{lemma}\label{lem: l.a.o. for S}
		Let $S$ be a cs $PL$ $(d-1)$-sphere that satisfies property $\mathrm{P}_{d-1}$ w.r.t. $(B, D, v)$. Let $W:=V(D)\uplus V(\st_{S_{d-2}}(v_{d-2}))$ and let
		$\mathbf{A}:=\{\{u,w\}\in S:\text{ } u\in W\text{ and }w\in \sigma(W)\}.$
	There exists an l.a.o.\ of $S$ such that:
		\begin{itemize}
			\item $u\rightarrow w$ for every edge $\{u,w\}$ in $\mathbf{A}$. 
			\item For every edge $\{u,w\}\in S$, $u\rightarrow w$ if and only if $\sigma(w)\rightarrow \sigma(u)$.
		\end{itemize}  
	\end{lemma}
	\begin{proof}
		First we choose any l.a.o.\ of the induced subcomplex of $S$ on $W$. Following \Cref{lem: symmetry l.a.o.} we impose on the induced subcomplex on $\sigma(W)$ a reverse orientation. Since $S$ satisfies Property $\mathrm{P}_{d-1}$, $V(S)=V(W)\uplus V(\sigma(W))$ and hence the edges of $S$ with a vertex in $V(W)$ and one in  $V(\sigma(W))$ belong to $\mathbf{A}$. We orient all edges $\{u,w\}$ in $\mathbf{A}$ as $u\rightarrow w$. This orientation is by definition acyclic on every $2$-simplex not containing edges in $\mathbf{A}$. It suffices to check that no edge in $\mathbf{A}$ is contained in an oriented cycle. Since every $2$-simplex containing an edge in $\mathbf{A}$ also contains another edge in $\mathbf{A}$, it contains a vertex $u$ or a vertex $w$ such that either $w_1\leftarrow u\rightarrow w_2$ or $u_1\rightarrow w\leftarrow u_2$. This proves the claim.
	\end{proof}
	
It follows from Lemmas \ref{lem: is PL manifold} and \ref{lem: l.a.o. for S} that $(S\times [-1,1])^\ell$ is cs PL triangulation of $\Sp^{d-1} \times [-1,1]$ for some l.a.o.\ $\ell$ of $S$. It is left to prove that this triangulation refines $\Sigma''$. 
	
	\begin{figure}[h]
		\centering
		\includegraphics[scale=0.8]{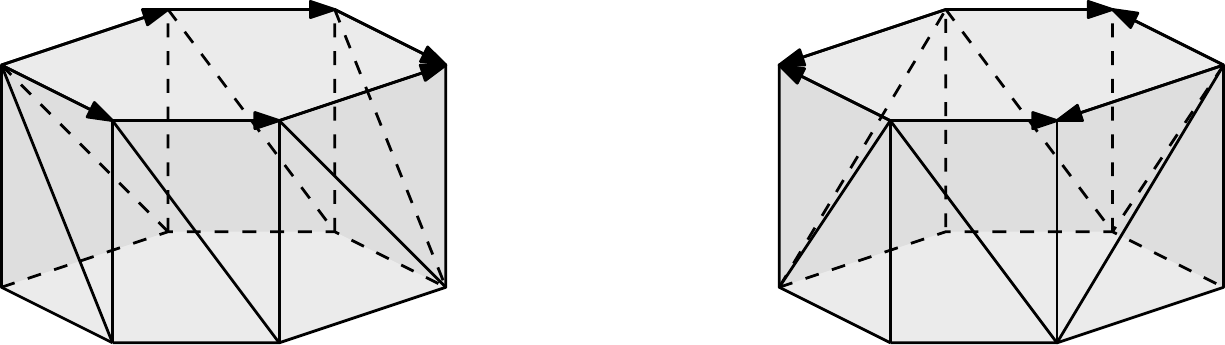}
		\caption{Two different locally acyclic orientations of the cs 6-cycle $S_1$ that satisfy \Cref{lem: l.a.o. for S} and induced triangulations of $S_1\times [-1, 1]$.}
		\label{fig: l.a.o.}
	\end{figure}

	\begin{corollary}\label{cor: exists triangulation prism}
		Let $S$ be a cs $(d-1)$-sphere that satisfies property $\mathrm{P}_{d-1}$ w.r.t. $(B, D, v)$. There exists an l.a.o.\ $\ell$ on $S$ such that 
		\begin{itemize}
			\item[i.] $(S\times[-1,1])^\ell$ is a cs PL manifold with boundary that refines $\Sigma''$;
			\item[ii.] $(S\times [-1,1])^\ell$ contains an induced cs subcomplex $\Gamma$ isomorphic to $S$ which contains $(D\times\{1\})\cup (\sigma(D)\times\{-1\})$;
			\item[iii.] $V((S\times [-1,1])^\ell)=V(\Sigma'')$.
		\end{itemize} 
	\end{corollary}
	\begin{proof}
		Choose any l.a.o.\ $\ell$ on $S$ that satisfies the conditions in Lemma \ref{lem: l.a.o. for S}. The orientation on any edge $\{u,w\}\in \mathbf{A}$ (as defined in Lemma \ref{lem: l.a.o. for S}) generates a new edge $\{(u,1), (w,-1)\}\in \Sigma$. This coincides with those edges in $\Sigma''$ but not in $S\times [-1,1]$. Hence $(S\times [-1,1])^\ell$ is a refinement of $\Sigma''$. By Lemma \ref{lem: is PL manifold}, $(S\times [-1,1])^\ell$ is a PL manifold. Part ii and iii follow from Lemmas \ref{lem: property of Sigma''} and \ref{lem: is PL manifold}.
	\end{proof}
	\begin{remark}As we see from \Cref{fig: l.a.o.}, the l.a.o.\ of $S$  that satisfies the conditions in \Cref{lem: l.a.o. for S} is not unique. In what follows, we denote by $\Sigma$ \emph{any} PL manifold $(S\times [-1, 1])^\ell$ obtained in \Cref{cor: exists triangulation prism}. Although $\Sigma$ can be defined directly from the l.a.o., Corollary \ref{cor: exists triangulation prism} ii (which follows from subsection 3.1) will play a key role in our inductive construction in Section 4.
	\end{remark}
	\subsection{From $\Sigma$ to a PL triangulation of $\Sp^d$}
	In this subsection we complete the cs triangulation of the prism over a PL $(d-1)$-sphere as in \Cref{cor: exists triangulation prism} to a PL $d$-sphere. We introduce two pairs of antipodal vertices $(v_{+},v_{-}),(w_{+},w_{-})$ and define
	\begin{equation}\label{eq: phi'}
	\Phi':= \Sigma \cup (v_{+} * K_+) \cup (v_{-} * K_-)\cup (w_{+} * L_+)\cup (w_{-} * L_-),
	\end{equation} 
	where 
	$K_{+}:=B\times\{1\}$, $K_{-}:=\sigma(B)\times\{-1\}$,  $L_{+}:=(\sigma(B)\times\{ 1\})\cup_{\partial B \times\{1\}} (v_{+}*\partial B\times\{1\})$, and $L_{-}:=(B\times\{ -1\})\cup_{\partial B \times\{-1\}} (v_{+}*\partial B\times\{-1\})$. The free involution $\sigma$ on $\Sigma$ extends to an involution on $\Phi'$ by additionally letting $\sigma(v_+)=v_-$ and $\sigma(w_+)=w_-$.
	The second complex in \Cref{fig: rp2} offers a visualization of $\Phi'$ in the $2$-dimensional case.
	\begin{proposition}
		The simplicial complex $\Phi'$ in \eqref{eq: phi'} is a cs PL triangulation of $\Sp^d$.
	\end{proposition}
	\begin{proof}
		By \Cref{cor: exists triangulation prism}, $\Sigma$ is a PL manifold with boundary. By Property $\mathrm{P}_{d-1}$, $K_{\pm}$ is a PL $(d-1)$-ball and by \Cref{lem: property of PL manifolds}, $L_{\pm}$ is a PL $(d-1)$-sphere. Again by \Cref{lem: property of PL manifolds}, $\Phi'$ is a PL manifold. Finally it is clear that $\Phi'$ is centrally symmetric and $|\Phi'|$ is homeomorphic to $\Sp^d$. 
	\end{proof}
	
	We next contract certain edges of $\Phi'$ in order to reduce the number of vertices. A simplicial complex is obtained from $\Delta$ via an \emph{edge contraction} of $\{i,j\}\in\Delta$ if it is the image of $\Delta$ w.r.t. the simplicial map that identifies $i$ with $j$. The edges we will contract are those of the form $\{(u,1),(u,-1)\}$, where $u\in V(D)\cup V(\sigma(D))$. In other words, we identify $D\times\{1\}$ and $\sigma(D)\times\{-1\}$ with $D\times\{-1\}$ and $\sigma(D)\times\{1\}$. It is easy to see that in this case the procedure does not depend on the order in which contractions are applied. To prove that the resulting complex is still a PL sphere we use a result of Nevo \cite{Nevo07}.
	\begin{theorem}[{\cite[Theorem 1.4]{Nevo07}}]\label{lem eran}
		Let $\Delta$ be a PL manifold and let $\Delta'$ be the contraction of $\Delta$ at the edge $\{i,j\}$. Then $\Delta'$ is PL homeomorphic to $\Delta$ if and only if $\lk_{\Delta}(i)\cap\lk_{\Delta}(j)=\lk_{\Delta}(\{i,j\})$.
	\end{theorem}

\begin{lemma}\label{lem: condition}
	Let $\Delta'$ be a simplicial complex and let $\Delta$ be a simplicial refinement of $\Delta'\times[-1,1]$ induced by an l.a.o.\ of $\Delta'$. Then for any vertex $u\in \Delta'$, $$\lk_{\Delta}((u,1))\cap\lk_{\Delta}((u,-1))=\lk_{\Delta}(\{(u,1),(u,-1)\}).$$
\end{lemma}
\begin{proof}
	The triangulation of the prism induced by an l.a.o.\ has the following key property: for every $F\in\Delta'$, a subset $G=\{(w_1,t_1),\dots,(w_k,t_k)\}$ of $V(F\times [-1,1])$ is a face of $\Delta$ if and only if any two vertices  $(w_i,t_i),(w_j,t_j)$ form an edge of $\Delta$. Therefore if both $G\cup \{(u,-1)\}$ and $G\cup \{(u,1)\}$ are faces of $\Delta$, then $G\cup \{(u,-1),(u,1)\}$ is also a face of $\Delta$. This proves that $\lk_{\Delta}((u,1))\cap\lk_{\Delta}((u,-1))\subseteq\lk_{\Delta}(\{(u,1),(u,-1)\})$. Since the other inclusion holds in general, the claim follows.
\end{proof}

We underline that \Cref{lem: condition} clearly does not imply that the simplicial complex obtained from $\Sigma$ by contracting $\{(u,1),(u,-1)\}$, $u\in S$, is PL homeomorphic to $\Sigma$. Indeed, \Cref{lem eran} holds for PL manifolds \emph{without boundary}.
	
	\begin{proposition}\label{cor: phi is sphere}
		The simplicial complex $\Phi$ obtained from $\Phi'$ by contracting every edge of the form $\{(u,1),(u,-1)\}$, $u\in V(D)\cup V(\sigma(D))$, is a cs PL $d$-sphere. 
	\end{proposition}
	\begin{proof}
		Let $V(D)=\{v_1, \dots, v_k\}$, $e_i=\{(v_i,1),(v_i,-1)\}$ for $i=1,\dots,k$. Let 
		$$\Phi'=\Phi_0\overset{e_1, \sigma(e_1)}{\longrightarrow}\Phi_1\overset{e_2, \sigma(e_2)}{\longrightarrow}\dots{\longrightarrow}\Phi_{k-1}\overset{e_{k},\sigma(e_{k})}{\longrightarrow}\Phi_k=\Phi$$
		be the sequence of complexes obtained from $\Phi'$ by contracting a pair of antipodal edges $e_i, \sigma(e_i)$ at a time. Since every vertex in $V(\Phi')\setminus V(\Sigma)$ is connected to at most one vertex from $\{(w, 1), (w,-1)\}$, where $w\in S$, we have that for every vertex $v\in S$, $$\lk_{\Phi'}((v,1))\cap\lk_{\Phi'}((v,-1))=\lk_{\Sigma}((v,1))\cap\lk_{\Sigma}((v,-1))$$
		$$\mathrm{and}\quad \lk_{\Phi'}(\{(v,1),(v,-1)\})=\lk_{\Sigma}(\{(v,1),(v,-1)\}). $$ By \Cref{lem eran}, \Cref{lem: condition}, and the fact that the links of $e_1$ and $\sigma(e_1)$ in $\Sigma$ are antipodal and disjoint, it follows that $\Phi_1$ is a cs PL $d$-sphere. 
Note that $e_i\cup e_j$ is not a face in $\Phi_0$ or $\Phi_1, \dots, \Phi_{i-1}$ for any distinct $i$ and $j$; in other words, at least one edge between $(v_j, \pm 1)$ and $(v_i, \pm 1)$ is missing. Hence $e_j \notin \lk_{\Phi_{k}}((v_i, 1))\cap \lk_{\Phi_{k}}((v_i, -1))$ for any $0\leq k\leq i-1$ and inductively, $$\lk_{\Phi_{i-1}}((v_i, 1))\cap \lk_{\Phi_{i-1}}((v_i, -1))=\lk_{\Phi_{i-2}}((v_i, 1))\cap \lk_{\Phi_{i-2}}((v_i, -1))=\dots =\lk_{\Phi_0}((v_i, 1))\cap \lk_{\Phi_0}((v_i, -1)).$$
Similarly, $e_j$ is not an edge in $\lk_{\Phi_{i-1}}(e_i), \lk_{\Phi_{i-2}}(e_i),\dots,\lk_{\Phi_{0}}(e_i)$ for any $i\neq j$ and hence $\lk_{\Phi_{i-1}}(e_i)=\lk_{\Phi_{0}}(e_i)$. This fact implies the first and third equalities in the following equation:
		\begin{align*}
		\lk_{\Phi_{i-1}}((v_{i},1))\cap\lk_{\Phi_{i-1}}((v_{i},-1))&= \lk_{\Phi_0}((v_{i},1))\cap\lk_{\Phi_0}((v_{i},-1))\\
		&=\lk_{\Phi_0}(e_i)\\
		&= \lk_{\Phi_{i-1}}(e_i),
		\end{align*}
		while the second equality follows from \Cref{lem eran}. Again by \Cref{lem eran} and the fact that $\lk_{\Phi_{i-1}}(e_i)$ and $\lk_{\Phi_{i-1}}(\sigma(e_i))$ are antipodal and disjoint, we conclude that every $\Phi_i$ is a cs PL $d$-sphere. 
	\end{proof}
	\begin{remark}\label{rem: more contractions}
		In fact, we can contract more edges $\{(u,-1),(u,1)\}$ and their antipodes with $u\in V(S)$ and still obtain a cs PL $d$-sphere. However, contracting too many edges would create induced cs 4-cycles in the resulting complex.
	\end{remark}
	\begin{remark}
		By Corollary \ref{cor: exists triangulation prism}, the cs $d$-sphere $\Phi$ that we define is usually not unique, but it depends on the l.a.o.\ considered. Indeed, the number of combinatorial types of $\Phi$ is related to the number of l.a.o.\ on $S$ that satisfy the conditions in Lemma \ref{lem: l.a.o. for S}. With a slight abuse of notation we write ``a sphere $\Phi$" to indicate any PL $d$-sphere that could be constructed from $S$ as in this section.
	\end{remark}
	
	\section{The induction step}
	In this section, we show that the sphere  $\Phi$ constructed in the previous section satisfies Property $\mathrm{P}_d$ w.r.t. a certain flag of spheres $S_0\subset S_1\subset \dots \subset S_{d-1}=S\subset S_d=\Phi$. Finally, we show that if $S_{d-1}$ does not have induced cs $4$-cycles then the same holds for $S_d$. Using this fact, together with the initial cases, i.e., the $0$-dimensional sphere $S_0$ and the cs 6-cycle $S_1$, we prove \Cref{thm: exists cs with no cycle}. Assume that inductively we've constructed a sequence of cs $i$-spheres $S_{i}$, $0\leq i\leq d-1$, that satisfies Property $\mathrm{P}_{i}$ under the triple $(B_{i}, D_{i}, v_{i})$, that is
	\begin{itemize}
		\item $B_i$ is a PL $i$-ball in $S_i$ with boundary $S_{i-1}$.
		\item $S_i\setminus S_{i-1}=D_i\uplus\sigma(D_i)$ for some PL $i$-ball $D$.
		\item $V(\st_{S_i}(v_i))\cup V(\st_{S_i}(\sigma(v_i)))=V(S_i)$.		
	\end{itemize}
	We will now show that the PL $d$-sphere $\Phi$ constructed in the previous section by setting $\Gamma$ to be the image of $S_{d-1}\times\{0\}$ w.r.t. the map $\psi$ in \eqref{eq: 3.1} (consequently, $\Gamma$ satisfies Property $\mathrm{P}_d$). With the notation introduced earlier we define:
	\begin{itemize}
		\item $D_d:= \{v_{+},w_{+}\}*\st_{S_{d-2}\times\{1\}}((\sigma(v_{d-2}),1))$.
		\item $B_d$ is the closure of one of the two connected components of $\Phi-\Gamma:=\{F\in \Phi\;|\;F\cap \Gamma=\emptyset\}$. 
		\item $v_d:=v_+$.
	\end{itemize}
	\begin{remark}
		By Jordan's theorem, the geometric realization of $\Phi -\Gamma$ consists of two connected components. Since $\Phi$ and $\Gamma$ are PL spheres, it is known that $B_d$ is a simplicial ball \cite[Theorem 6]{Newm}. However, it is an open problem in PL topology (known as PL Schoenflies problem) to decide whether $B_d$ is also a PL ball. The following lemma gives a positive answer in the special case of our construction.
	\end{remark}
	\begin{figure}[h]
		\centering
		\includegraphics[scale=0.9]{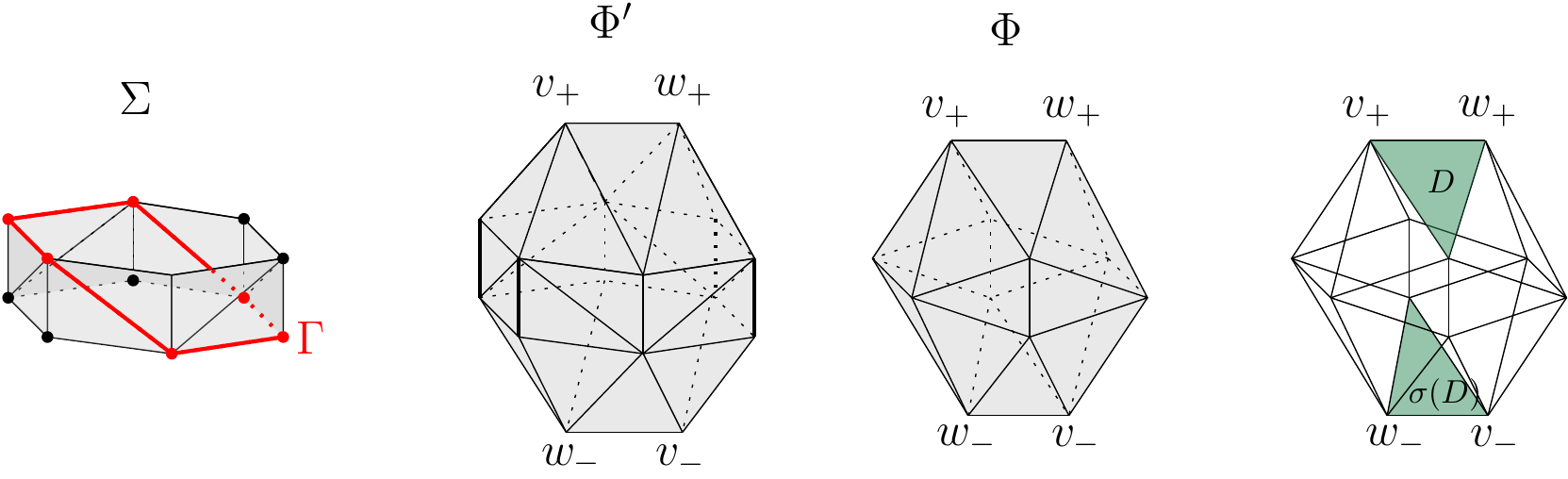}
		\caption{An illustration of \Cref{thm: main} in the case $d=2$}
		\label{fig: rp2}
	\end{figure}
	\begin{lemma}
		The triple $(B_d, D_d, v_d)$ satisfies the following properties: 
		\begin{itemize}
			\item $B_d$ is a PL $d$-ball in $\Phi$ with boundary $\Gamma\cong S_{d-1}$.
			\item $D_d$ is isomorphic to $\{v_d, w_{+}\}* \st_{S_{d-2}}(v_{d-2})$. 
			\item $D_d\subseteq \st_{\Phi}(v_d) \subseteq B_d$ and  $V(\Phi)=V(\st_{\Phi}(v_d))\uplus V(\st_{\Phi}(\sigma(v_d)))$.
		\end{itemize}
		In particular, $\Phi$ satisfies Property $\mathrm{P}_d$ under the triple $(B_d, D_d, v_d)$.
	\end{lemma}
	\begin{proof}
		We only need to verify the first and third bullet points. 
		By the inductive hypothesis and the definition of $D_d$, we obtain
		\begin{equation*}
		\begin{split}
		V(\st_{ \Phi}(v_d))&=V(B_{d-1}\times \{1\})\cup\{v_d, w_{+}\}\\
		&=\Big( V(D_{d-1}\times \{1\})\cup V(\st_{S_{d-2}\times \{1\}}(v_{d-2},1))\Big)  \cup \Big( V(\st_{S_{d-2}\times \{1\}}(\sigma(v_{d-2}), 1))\cup\{v_d, w_{+}\}\Big) \\
		&=V(\st_{\Gamma}(\sigma(v_{d-1})))\cup  V(D_{d}).
		\end{split}
		\end{equation*}
		Therefore, 
		\begin{align*}
		V(\st_{\Phi}(v_d))\uplus V(\st_{\Phi}(\sigma(v_d)))&=V(\st_{\Gamma}(v_{d-1}))\uplus V(\st_{\Gamma}(\sigma(v_{d-1})))\uplus V(D_d) \uplus V(\sigma(D_d))\\
		&= V(\Gamma)\uplus V(D_d) \uplus V(\sigma(D_d))\\
		& = V(\Phi),
		\end{align*}
		where the second equality follows from \Cref{def: P_d} iii on $\Gamma$.
		
		To see that $B_d$ is a PL $d$-ball, we consider the simplicial complex $\Phi^*$ obtained from $\Phi$ by contracting all the edges of the form $\{(v,-1),(v,1)\}$ with $v\in V(\st_{\Phi}(v_d))$. By \Cref{lem: condition} and \Cref{lem eran}, $\Phi^*$ is a PL $d$-sphere. The image of $\Sigma \cup (v_{+} * K_+)\cup (w_{+} * L_+)$ is a PL $d$-ball, since $K_+$ and $L_+$ are a PL $(d-1)$-ball and a $(d-2)$-sphere respectively. It follows from \Cref{lem: newmann} that $B_d$ is a PL $d$-ball, since it can be decomposed as the (closure of) the complement of the PL $d$-ball $\Sigma \cup (v_{+} * K_+)\cup (w_{+} * L_+)$ w.r.t. the PL $d$-sphere $\Phi^*$.
		
		
		Finally, by the definition, we have that $B_d\setminus \Gamma=D_d$ and $D_d\subseteq \st_\Phi(v_d)\subseteq B_d$. Furthermore, $V(\st_\Phi(v_d))\cup V(\st_\Phi(\sigma(v_d))=V(\Phi)$ follows from the inductive assumption that $V(\st_\Gamma(v_{d-1}))\cup V(\st_\Gamma(\sigma(v_{d-1})))=V(\Gamma)$.
	\end{proof}
	
	The property which motivates our construction is the content of the following lemma.
	\begin{proposition}\label{prop: no induced 4}
		The complex $\Phi$ as constructed above has no induced cs $4$-cycle if $\Gamma\cong S_{d-1}$ has no induced cs $4$-cycle.
	\end{proposition}
	\begin{proof}
		As we see from the construction, since $\lk_{\Phi}(v_+)\cap \lk_{\Phi}(v_-)=\lk_{\Phi}(w_+)\cap \lk_{\Phi}(w_-)=\{\emptyset\}$, it follows that none of the vertices $v_+, v_-, w_+,w_-$ belong to any induced cs $4$-cycle. Furthermore, any vertex $(a,1)\in S_{d-2}\times\{1\}$ is only adjacent to either  $v_+, w_+$, or its neighbors in $S_{d-1}\times\{1\}$, or some $(\sigma(b),-1)\in S_{d-2}\times\{-1\}$ where $\{\sigma(a),\sigma(b)\}\in S_{d-2}\times\{-1\}$. By the inductive hypothesis, $S_{d-2}$ and $S_{d-1}$ do not contain any induced cs $4$-cycle. We conclude that $(a,1)$ is also not in any induced cs $4$-cycle in $S$. However, any cs $4$-cycle must include two vertices outside $D_d\cup\sigma(D_d)$. Hence $\Phi$ has no induced cs $4$-cycle.
	\end{proof}
	
	Finally, we conclude with the proof of the main result.
	\begin{theorem}\label{thm: main}
		There exists a family of cs PL $i$-spheres $S_0\subset S_1 \subset S_2\dots$ such that each $S_i$ satisfies Property $\mathrm{P}_i$ with respect to $(B_i, D_i, v_i)$, where $\partial B_i=S_{i-1}$, and each $S_i$ has no induced cs 4-cycles. Furthermore $f_0(S_{i+1})=f_0(S_i)+f_0(S_{i-1})+4$ for $i\leq 1$.
	\end{theorem}
	\begin{proof}
		The first statement follows directly from \Cref{cor: phi is sphere} and \Cref{prop: no induced 4}. The number of vertices in the prism over $S_{i-1}$ equals $2f_0(S_{i-1})$, and together with $v_+, v_-, w_+,w_-$ sums up to $2f_0(S_{i-1})+4$. Identifying the vertices of $D_{i-1}\times\{\pm 1\}$ and $\sigma(D_{i-1})\times\{\pm 1\}$ decreases the number of vertices by $2f_0(D_{i-1})$. Since $f_0(S_{i-1})=2f_0(D_{i-1})+f_0(S_{i-2})$, the claim follows.
	\end{proof}
	
	\smallskip\noindent {\it Proof of \Cref{thm: exists cs with no cycle}:\ } Via \Cref{thm: main} we know there exists a PL $d$-sphere $S_d$ whose number of vertices $n_d$ is given by the sequence $n_0=2$, $n_1=6$, $n_{i+1}=n_{i}+n_{i-1}+4$. Solving the recursion we obtain the desired formula. The second statement  follows from a direct application of \Cref{lem: from sphere to rpd}.\hfill$\square$\medskip

	\begin{remark}
	Our inductive method produces the double cover of the minimal triangulation of $\R P^2$, the boundary of the icosahedron, from the double cover of the minimal triangulation of $\R P^1$, the 6-cycle; see Figure \ref{fig: rp2}. In dimension 3, we find two non-isomorphic PL triangulations of $\R P^3$ with the $f$-vector $(11,52,82,41)$, starting from the boundary of the icosahedron. They are vertex-minimal but not $f$-vectorwise minimal triangulations. For $d=4,5$, our construction is not vertex-minimal. We developed a naive implementation of our construction in the software {\tt Sage} and generated triangulations of $\R P^d$ as in \Cref{thm: main} up to $d=7$, and compute standard invariants (e.g., fundamental group, homology groups and pseudomanifold property) up to $d=6$. The code and the lists of facets of these triangulations can be found in \cite{GitRepo}. In \Cref{table} we report their $f$-vectors.
	\begin{table}[h!]
		\begin{center}
			\vspace{0pt}
			\begin{tabular}{  l | l }
				$d$ & $f$-vector \\ \hline
				1 & (3,3) \\
				2 & (6, 15, 10) \\
				3 & (11, 52, 82, 41)\\
				4 & (19, 151, 424, 485, 194)\\
				5 & (32, 403, 1797, 3536, 3165, 1055)\\
				6 & (53, 1022, 6811, 20545, 30919, 22701, 6486) \\
				7 & (87, 2514, 24099, 104628, 235599, 286041, 177864, 44466)
			\end{tabular}
		\end{center}
		\caption{$f$-vectors of a triangulation of  $\R P^d$ as in  \Cref{thm: exists cs with no cycle}.}
		\label{table}
	\end{table}

	\end{remark}
	\section{Open problems}
	We conclude this article with a few questions. The first one is about the (asymptotic) tight lower bound on the number of vertices required for a vertex-minimal triangulation of $\R P^d$.
	\begin{question}
		Does there exist a PL triangulation of $\mathbb{R}P^d$ with $\binom{d+2}{2}+\lfloor\frac{d-1}{2}\rfloor$ vertices for every $d\geq 1$? Does at least a construction with a number of vertices that is polynomial in $d$ exist?
	\end{question}
	We do not know if the PL spheres constructed in \Cref{thm: main} are polytopal, i.e., they can be realized as the boundary complex of a simplicial polytope. It is natural to ask the following question.
	\begin{question}
		What is the minimum number of vertices required for a cs simplicial $d$-polytope with no induced cs 4-cycles?
	\end{question}
	Frequently in the literature, additional combinatorial properties are imposed on a triangulation. We focus on two properties, namely flagness and balancedness. A simplicial complex is \emph{flag} if all minimal subsets of the vertices which do not form a face are edges. A $d$-dimensional simplicial complex is \emph{balanced} if there exists a simplicial projection (often called coloring) to the $d$-simplex which preserves the dimension of faces. Recently Bibby et al. \cite{BOWWZZ} and the first author in \cite{Ven} implemented local flips and transformations in order to obtain flag and balanced triangulations of manifolds which are vertex-minimal w.r.t. these properties. In particular the authors obtained a flag and balanced vertex-minimal triangulation of $\mathbb{R}P^2$ on $11$ and $9$ vertices respectively, and a balanced vertex-minimal triangulation of $\mathbb{R}P^3$ on $16$ vertices. For higher $d$, a flag and balanced triangulation of $\mathbb{R}P^d$ can be obtained by considering the barycentric subdivision of the boundary complex of the $(d+1)$-dimensional \emph{cross-polytope}, and identifying antipodal vertices. These simplicial complexes have $\frac{3^{d+1}-1}{2}$ vertices.
	\begin{problem}
		Construct flag or balanced PL triangulations of $\mathbb{R}P^d$ for every $d$ with less then $\frac{3^{d+1}-1}{2}$ vertices. Does a construction on a number of vertices that is polynomial in $d$ exist? 
	\end{problem}
	\section*{Acknowledgements}
	We would like to thank Basudeb Datta, Isabella Novik and Martina Juhnke-Kubitzke for helpful comments. We also would like to express our gratitude to the anonymous referees for providing references for results on minimal triangulated manifolds in other settings, as well as an insightful remark on the properness of group actions on PL manifolds. Their detailed feedback greatly helped us to improve earlier versions of the article. 

	{\small
		\bibliography{refs}
		\bibliographystyle{alpha}
	}
\end{document}